\documentclass[a4paper,11pt]{amsart}

\usepackage[utf8]{inputenc}
\usepackage{amsmath,amsfonts,amscd,amssymb,amsthm,epsfig,euscript}
\usepackage{mathrsfs}
\usepackage{amsxtra}
\usepackage{hyperref}
\hypersetup{colorlinks,citecolor=blue,urlcolor=blue,plainpages=false,hypertexnames=false}

\newtheorem{thm}{Theorem}
\newtheorem{lem}{Lemma}
\newtheorem{pro}{Proposition}
\newtheorem{cor}{Corollary}
\theoremstyle{remark}
\newtheorem{remark}{Remark}
\numberwithin{equation}{section}
\usepackage[all]{xy}

\newcommand{\cal}{\mathcal}

\title{On a functional of Kobayashi for Higgs bundles}

\author[Cardona]{Sergio A. H. Cardona}
\address{Conacyt Research Fellow--Instituto de Matem\'aticas, Universidad Nacional \indent Aut\'onoma de 
M\'exico, Leon 2 altos, Col. centro, 68000, Oaxaca, Mexico}
\email{sholguin@im.unam.mx} 

\author[Meneses]{Claudio Meneses}
\address{Research Fellow--Mathematisches Seminar, Christian-Albrechts Universit\"at \indent zu Kiel, Ludewig-Meyn-Str. 4, 
24118 Kiel, Germany}
\email{meneses@math.uni-kiel.de} 


\begin{document}

\maketitle 

\begin{abstract}
We define a functional ${\cal J}(h)$ for the space of Hermitian metrics on an arbitrary Higgs bundle over a compact K\"ahler manifold, as a natural generalization of the mean curvature energy functional of Kobayashi for holomorphic vector bundles  \cite{Kobayashi}, and study some of its basic properties. We show that ${\cal J}(h)$ is bounded from below by a nonnegative constant depending on invariants of the Higgs bundle and the K\"ahler manifold, and that when achieved, its absolute minima are Hermite-Yang-Mills metrics. We derive a formula relating ${\cal J}(h)$ and another functional ${\cal I}(h)$, closely related to the Yang-Mills-Higgs functional \cite{Bradlow-Wilkin, Wentworth}, which can be thought of as an extension of a formula of Kobayashi for holomorphic vector bundles to the Higgs bundles setting. Finally, using 1-parameter families in the space of Hermitian metrics on a Higgs bundle, we compute the first variation of ${\cal J}(h)$, which is expressed as a certain $L^{2}$-Hermitian inner product. It follows that a Hermitian metric on a Higgs bundle is a critical point of ${\cal J}(h)$ if and only if the corresponding Hitchin--Simpson mean curvature is parallel with respect to the Hitchin--Simpson connection. \\


\end{abstract}

%%%%%%%%%%%%%%%%%%%%%%%%%%%%%%%%%%%%%%%%%%%%%%%%%%%%%%%%%%%%%%%%%%%%%%%%%
%%%%%%%%%%%%%%%%%%%%%%%%%%%%%%%%%%%%%%%%%%%%%%%%%%%%%%%%%%%%%%%%%%%%%%%%%

\section{Introduction}

Since their introduction in the early 1980s, Hermite--Yang--Mills metrics have played a fundamental role in multiple subsequent developments in the fields of complex and algebraic geometry \cite{Buchdahl-0, Kobayashi, Siu, Uhlenbeck-Yau}. Kobayashi \cite{Kobayashi 0, Kobayashi} introduced Hermite--Yang--Mills metrics on holomorphic vector bundles over compact K\"ahler manifolds (that he named Hermite--Einstein), as a natural generalization of K\"ahler--Einstein metrics in the tangent bundle of a compact K\"ahler manifold. The fundamental result, that is generically known as {\it the Hitchin--Kobayashi correspondence}, states the equivalence between  the Mumford--Takemoto stability of a vector bundle---an algebraic notion---and the existence of Hermite--Yang--Mills metrics on it, which can be realised as critical points of the Donaldson and Yang--Mills functionals (see \cite{Kobayashi, Lubke} for details). The Hitchin--Kobayashi correspondence was first established in the pionnering works of Donaldson \cite{Donaldson-1, Donaldson-2, Donaldson-3}, Kobayashi \cite{Kobayashi} and Uhlenbeck-Yau \cite{Uhlenbeck-Yau} for holomorphic vector bundles over compact K\"ahler manifolds, and has also been extended to more general contexts, including holomorphic vector bundles over arbitrary compact complex manifolds \cite{Lubke}, and reflexive sheaves over compact K\"ahler manifolds \cite{Bando-Siu}. In fact, Kobayashi \cite{Kobayashi} also noted that for holomorphic bundles on a compact K\"ahler manifold, the $L^{2}$-energy of the mean curvature for the Chern connection of a Hermitian metric 
$h$ leads to yet another functional $J(h)$, whose absolute minima are precisely Hermite--Yang--Mills metrics, whenever they exist. Such a result is closely tied to the more standard descriptions, since it is verified that $J(h)$ differs from the $L^{2}$-energy of the curvature form of the Chern connection of $h$ by a constant depending only on the first and second Chern classes of the vector bundle, as well as the cohomology class of the K\"ahler form. Consequently, studying the functional $J(h)$ turns out to be equivalent to studying the Yang-Mills functional $I(h)$ under a choice of holomorphic structure.\\

In another direction, following the ideas of Narasimhan--Seshadri \cite{Narasimhan-Seshadri} on Mumford stability, and of Atiyah--Bott \cite{Atiyah79,Atiyah-Bott} on Yang--Mills theory, Hitchin \cite{Hitchin} proved a Hitchin-Kobayashi correspondence for pairs $(V,\phi)$ consisting of a rank 2 holomorphic vector bundle over a compact Riemann surface $V\longrightarrow X$ and a morphism of vector bundles $\phi: V\longrightarrow V\otimes\Omega^{1,0}_{X}$, commonly called a Higgs field. Simpson \cite{Simpson} extended Hitchin's results to more general Higgs pairs $(E,\phi)$ of arbitrary rank over a compact K\"ahler manifold $X$, where $\phi:E\longrightarrow E\otimes\Omega^{1,0}_{X}$ is also constrained to satisfy $\phi\wedge\phi=0$.\footnote{The result of Simpson is true even for some non-compact K\"ahler manifolds. Later on, this fact was used by Bando and Siu \cite{Bando-Siu} in order to prove the Hitchin--Kobayashi correspondence for reflexive sheaves.} Using the ideas of Bando and Siu \cite{Bando-Siu}, Biswas and Schumacher \cite{Biswas-Schumacher} extended this correspondence for reflexive Higgs sheaves over compact K\"ahler manifolds. Following the ideas of Donaldson, Simpson  introduced in \cite{Simpson} a Donaldson functional for Higgs bundles over compact (and some non-compact) K\"ahler manifolds. Such a functional can be also introduced using the approach of Kobayashi (for details on this see \cite{Cardona 1, Cardona 2}). However, as far as the authors know, an extension of the Kobayashi functional $J(h)$ to Higgs bundles has not been studied yet in the literature. Since for holomorphic vector bundles $J(h)$ is closely related to the Yang-Mills functional, an extension of $J(h)$ to Higgs bundles is {\it a priori} a functional of interest in complex geometry. The main purpose of this article is to introduce a natural extension for Higgs bundles of the mean curvature energy functional of Kobayashi---which we denote by  ${\cal J}(h)$---and to study some of its basic properties. In particular, we will see that in analogy to the classical holomorphic vector bundle case, the new functional ${\cal J}(h)$ is bounded from below by a nonnegative constant, and that its absolute minima (if achieved) are Hermite--Yang--Mills metrics. Therefore, as a consequence of the Hitchin--Kobayashi correspondence, we conclude that ${\cal J}(h)$ attains its lower bound if and only if the Higgs bundle is Mumford--Takemoto polystable. Using the $L^{2}$-norm of the $(1,1)$-part of the Hitchin--Simpson curvature, we then propose a natural candidate extension ${\cal I}(h)$ of the curvature energy functional $I(h)$, which is closely related to the Yang-Mills-Higgs functional \cite{Bradlow-Wilkin, Wentworth} under a choice of holomorphic structures that define $\mathfrak{E}$, and which is also closely related to ${\cal J}(h)$. However, the formula estimating the difference ${\cal J}(h) - {\cal I}(h)$ is no longer a constant as is the case for $J(h) - I(h)$, but now contains an additional term involving the Higgs field and the Hitchin--Simpson curvature, indicating a fundamental difference between ${\cal J}(h)$ and the Yang--Mills--Higgs functional. We believe that such a difference provides supporting evidence to motivate the study of further features of  ${\cal J}(h)$, such as the description of its critical points, which we perform using curves in the space of Hermitian metrics. In particular, we show that a Hermitian metric is a critical point of ${\cal J}(h)$ if and only if its Hitchin--Simpson mean curvature is parallel with respect to the Hitchin--Simpson connection. Such a result can be seen as a natural extension to Higgs bundles of a result of Kobayashi \cite{Kobayashi} for  $J(h)$ in the classical holomorphic vector bundle case.\footnote{Kobayashi proved in \cite{Kobayashi} that a Hermitian metric is a critical point of $J(h)$ if and only if the corresponding Chern mean curvature is parallel with respect to the Chern connection.}\\

This article is organized as follows. In Section 2 we review some basic facts about Higgs bundles and Hermitian metrics that will be used subsequently. In Section 3 we define the functional ${\cal J}(h)$ for Higgs bundles over compact K\"ahler manifolds as a natural extension of the Kobayashi functional $J(h)$. We prove that in strictly analogy with the holomorphic vector bundle case, ${\cal J}(h)$  is bounded from below by a nonnegative constant that depends on invariants of the bundle and the base manifold, and that if achieved, its absolute minima are Hermite--Yang--Mills metrics.   We then define a functional ${\cal I}(h)$ for Higgs bundles as an extension of Kobayashi's $I(h)$ functional \cite{Kobayashi}, which is closely related to the Yang--Mills--Higgs functional, and we obtain a formula measuring the difference between ${\cal J}(h)$ and ${\cal I}(h)$. Finally and following the ideas of Kobayashi \cite{Kobayashi}, in Section 4 we study some evolution properties of ${\cal J}(h)$, and we prove that the critical points of ${\cal J}(h)$ are closely tied to a parallelism condition of the Hitchin-Simpson mean curvature.\\

\noindent{\bf Acknowledgements}. Part of this article was done during a stay of the first author at Centro de Investigaci\'on en Matem\'aticas (CIMAT) in Mexico. The authors want to thank CIMAT for the hospitality. The first author was partially supported by the CONACyT grant 256126. The second author was partially supported by the DFG SPP 2026 priority programme ``Geometry at infinity".

%%%%%%%%%%%%%%%%%%%%%%%%%%%%%%%%%%%%%%%%%%%%%%%%%%%%%%%%%%%%%%%%%%%%%%%%%%
%%%%%%%%%%%%%%%%%%%%%%%%%%%%%%%%%%%%%%%%%%%%%%%%%%%%%%%%%%%%%%%%%%%%%%%%%%

\section{Preliminaries}

For convenience, in this section we will fix notation and review some basic definitions for holomorphic and Higgs bundles. Further details can be found in \cite{Bruzzo-Granha, Cardona 1, Kobayashi} and \cite{Simpson}.\\ 

Let $X$ be a compact $n$-dimensional K\"ahler manifold, with K\"ahler metric $g$, K\"ahler form $\omega$, and volume given in terms of the integral ${\rm Vol\,}X = \int_{X}\omega^{n}/n!$. Let $\Omega_{X}^{1,0}$ and $\Omega_{X}^{0,1}$ denote the holomorphic and anti-holomorphic cotangent bundles of $X$ and let $\Omega_{X}^{p,q}$ be the bundle over $X$ obtained by taking ($p$ and $q$ times) wedge products of $\Omega_{X}^{1,0}$ and $\Omega_{X}^{0,1}$. Sections of $\Omega_{X}^{p,q}$ are forms of type $(p,q)$ over $X$, and the space of all these forms is usually denoted by $A_{X}^{p,q}$. A Higgs bundle ${\mathfrak E}$ over $X$ is a pair $(E,\phi)$, where $E$ is a holomorphic vector bundle over $X$ of rank $r \geq 1$ and $\phi : E\rightarrow E\otimes\Omega_{X}^{1,0}$ is a holomorphic $\textrm{End}\, E$-valued $(1,0)$-form---the Higgs field---such that $\phi\wedge\phi : E\rightarrow E\otimes\Omega_{X}^{2,0}$ vanishes.  We define the degree ${\rm deg\,}{\mathfrak E}$ and slope $\mu({\mathfrak E})$ of ${\mathfrak E}$, relative to the K\"ahler form $\omega$, as the degree and slope of the underlying holomorphic vector bundle $E$,
\begin{equation}
{\rm deg\,}{\mathfrak E} = {\rm deg\,}{E} =  \int\limits_{X}c_{1}(E)\wedge\omega^{n-1}\,,\quad\quad  \mu({\mathfrak E})={\rm deg\,}{\mathfrak E}/r\,, \label{def. deg}
\end{equation}
where $c_{1}(E)$ denotes an arbitrary 2-form representing the first Chern class of $E$.
A Higgs subsheaf of ${\mathfrak E}$ is  a $\phi$-invariant subsheaf $F$ of $E$, leading to the Higgs pair ${\mathfrak F}=(F,\phi\lvert_{F})$. Fixing a given slope, the induced algebraic notion of stability for Higgs bundles, commonly called Mumford--Takemoto stability or $\mu$-stability, is defined as follows. We say that a Higgs bundle ${\mathfrak E}$ is \emph{$\mu$-stable} (resp. \emph{$\mu$-semistable}) if for any proper and non-trivial Higgs subsheaf ${\mathfrak F}$ of ${\mathfrak E}$ one has $\mu({\mathfrak F})<\mu({\mathfrak E})$ (resp. $\le$). A Higgs bundle is said to be {\it $\mu$-polystable} if it is a direct sum of a $\mu$-stable Higgs bundles with equal slopes.\\

A Hermitian Higgs bundle is a pair $({\mathfrak E},h)$ where $h$ is a Hermitian metric on the underlying vector bundle $E$. The Hermitian metric $h$ defines a Cartan involution on the fibers of the endomorphism bundle $\mathrm{End}\, E = E\otimes E^{\vee}$ and induces a real splitting into skew-Hermitian and Hermitian eigenbundles 
\begin{equation}
\mathrm{End}\, E = (\mathrm{End}\, E)_{+}\oplus (\mathrm{End}\, E)_{-}. \label{eq. splitting}
\end{equation}
A complex connection $D$ in $E$ is called \emph{unitary} if it preserves the splitting \eqref{eq. splitting}, or equivalently, if it satisfies 
\begin{equation}
dh(s,s')=h(Ds,s') + h(s,Ds') \nonumber
\end{equation} 
for all smooth sections $s,s'$ of $E$. For any $\mathrm{End}\,E$-valued $(p,q)$-form $\psi$, its {\it adjoint} $\psi^{*}_{h}$ is defined as the $\mathrm{End}\,E$-valued $(q,p)$-form satisfying
\[
h(\psi^{*}_{h}s,s') = h(s,\psi s')
\]
for all sections $s,s'$ of $E$. In particular, the adjoint of the Higgs field $\phi$ is an anti-holomorphic $\mathrm{End}\,E$-valued $(0,1)$-form $\phi^{*}_{h}: E\longrightarrow E\otimes\Omega^{0,1}_{X}$, and the condition $\phi\wedge\phi=0$ implies that the identity $\phi^{*}_{h}\wedge\phi^{*}_{h}=0$ is also satisfied. Let $D_{h}=D'_{h} + D''$ be the Chern connection of the pair $(E,h)$, i.e., the unique unitary connection on $E$ satisfying  $D''=d''_{E}$, where $D'_{h}$ and $D''$ are its $(1,0)$ and $(0,1)$ components, respectively.\footnote{The notation is chosen to indicate the explicit dependence of the Chern connection on $h$, which is present in its $(1,0)$ part only.}  
Combining $D_{h}$, $\phi$ and $\phi^{*}_{h}$, Simpson \cite{Simpson} defines the operators ${\cal D}'_{h}=D'_{h} + \phi^{*}_{h}$ and ${\cal D}''=D''+\phi$. Their sum determines yet another connection ${\cal D}_{h}$, the so-called \emph{Hitchin-Simpson connection}
\begin{equation}
{\cal D}_{h} = {\cal D}'_{h} + {\cal D}'' 
= D_{h} + \phi + \phi^{*}_{h}\,. \label{HS-connection}
\end{equation}
The operators ${\cal D}'_{h}$ and ${\cal D}''$ do not define the splitting of ${\cal D}_{h}$ into type, but rather indicate its dependance on $h$. By definition, the {\it Hitchin-Simpson curvature} is the curvature of ${\cal D}_{h}$ and hence ${\cal R}_{h}={\cal D}_{h}\circ {\cal D}_{h}$, and corresponds to an $\mathrm{End}\, E$-valued $2$-form $\Omega$. From 
(\ref{HS-connection}) we get a formula for ${\cal R}_{h}$ in terms of the components of the Chern connection $D_{h} $ and its curvature $R_{h}=D_{h}\circ D_{h}$ as follows
\begin{eqnarray*}
{\cal R}_{h} &=& {\cal R}_{h}^{2,0} + {\cal R}_{h}^{1,1} + {\cal R}_{h}^{0,2}\\
                   &=& (D'_{h}\phi + \phi\wedge\phi) + \left(R_{h} + [\phi,\phi^{*}_{h}]\right) + \left(d''_{E}\phi + D'_{h}\phi^{*}_{h}\right)\\
                   && + \, \left(d''_{E}\phi^{*}_{h} + \phi^{*}_{h}\wedge\phi^{*}_{h}\right) \\
                   &=& D'_{h}\phi + \left(R_{h} + [\phi,\phi^{*}_{h}]\right) + \left(d''_{E}\phi + D'_{h}\phi^{*}_{h}\right) + d''_{E}\phi^{*}_{h}
\end{eqnarray*}
where $D'_{h}$ and $d''_{E}$ denote induced covariant derivatives on a corresponding vector bundle and 
\begin{equation}
[\phi,\phi^{*}_{h}]=\phi\wedge\phi^{*}_{h} + \phi^{*}_{h}\wedge\phi \nonumber
\end{equation} 
is the Lie bracket extension to $\mathrm{End}\, E$-valued 1-forms. Moreover, the subcomponents 
\begin{equation}
{\cal R}^{1,1}_{+} = R_{h} + [\phi,\phi^{*}_{h}]\,, \quad\quad\quad {\cal R}^{1,1}_{-} = d''_{E}\phi + D'_{h}\phi^{*}_{h}\,, \nonumber
\end{equation}
correspond to the bundle splitting \eqref{eq. splitting}. Since $\phi$ is holomorphic and $\phi^{*}_{h}$ is anti-holomorphic, it follows that ${\cal R}^{1,1}_{-}\equiv 0$. Therefore, the (1,1)-part of ${\cal R}_{h}$ is equal to its unitary component, i.e., it is the section of $\Omega^{1,1}({\rm End\,}E)$ given by 
\begin{equation}
{\cal R}^{1,1}_{h} = {\cal R}^{1,1}_{+}  = R_{h} +  [\phi,\phi^{*}_{h}]\,. \label{Formula eff. HS curv.}
\end{equation}
Following \cite{Bruzzo-Granha}, let us consider the operator $L:\Omega^{p,q}_{X}\longrightarrow\Omega^{p+1,q+1}_{X}$ defined by $L\psi=\omega\wedge\psi$ and its $g$-adjoint $\Lambda:\Omega^{p,q}_{X}\longrightarrow\Omega^{p-1,q-1}_{X}$. The {\it Hitchin-Simpson mean curvature} and {\it Hitchin-Simpson scalar curvature} are respectively given by
\begin{equation}
{\cal K}_{h} = \sqrt{-1}\Lambda{\cal R}_{h} = \sqrt{-1}\Lambda{\cal R}^{1,1}_{h}\,,\quad\quad\quad \sigma_{h} = {\rm tr\,}{\cal K}_{h}\,. \label{def. K,sigma}
\end{equation}
Hence, ${\cal K}_{h}$ is an endomorphism of $E$ and $\sigma_{h}$ is just its trace. As it is well known \cite{Kobayashi, Siu}, the curvature of the Chern connection and the Chern mean curvature $K_{h} := \sqrt{-1}\Lambda R_{h}$ are related by 
\begin{equation}
K_{h}\,\omega^{n} = \sqrt{-1}n\,R_{h}\wedge\omega^{n-1}\,. \nonumber
\end{equation}
For Hermitian Higgs bundles we have a similar relation but involving the Hitchin-Simpson curvature and ${\cal K}_{h}$ (see \cite{Cardona 1} for details). To be precise
\begin{equation}
{\cal K}_{h}\,\omega^{n} = \sqrt{-1}n\,{\cal R}_{h}\wedge\omega^{n-1}\,. \label{Formula R and K}
\end{equation}
The formula (\ref{Formula R and K}) could also be regarded as a definition of the Hitchin-Simpson mean curvature.\\

In terms of local holomorphic coordinates in a neighborhood $\mathscr{U}\subset X$ and a local unitary frame $\{e_{i}\}_{i=1}^{r}$ of $E|_{\mathscr{U}}$, with dual frame $\{e^{i}\}_{i=1}^{r}$, we have
\begin{equation}
\omega|_{\mathscr{U}} = \sqrt{-1}\sum_{\alpha,\beta} g_{\alpha\bar{\beta}}dz^{\alpha}\wedge d\bar{z}^{\beta}\,, \quad\quad\quad h|_{\mathscr{U}} = \sum_{i} e^{i}\otimes\bar{e}^{i} \nonumber
\end{equation}
and we can express 
\begin{eqnarray}
{\cal R}_{h}^{1,1}|_{\mathscr{U}} &=& \sum_{i,j,\alpha,\beta}\left({\cal R}_{h}^{1,1}\right)^{i}_{j\alpha\bar\beta}\,e_{i}\otimes e^{j}\,dz^{\alpha} \wedge d\bar{z}^{\beta}\,, \label{Comp. cal R11}\\
{\cal K}_{h}|_{\mathscr{U}} &=& \sum_{i,j} ({\cal K}_{h})^{i}_{j}\,e_{i}\otimes e^{j}\,, \label{Comp. cal K}
\end{eqnarray} 
where
\begin{equation}
({\cal K}_{h})^{i}_{j} = \sum_{\alpha,\beta}g^{\alpha\bar\beta}\left({\cal R}_{h}^{1,1}\right)^{i}_{j\alpha\bar\beta}\,.  \nonumber
\end{equation}
The unitary condition of ${\cal R}^{1,1}_{h}$ can also be concluded from the following computation. From standard literature on complex geometry \cite{Kobayashi, Siu} we know that the curvature of the Chern connection of the pair $(E,h)$ satisfies the relations 
\begin{equation}
\overline{(R_{h})^{i}_{j\alpha\bar\beta}}=(R_{h})^{j}_{\,i\beta\bar\alpha}. \nonumber
\end{equation} 
Now, writting 
\begin{equation}
\phi|_{\mathscr{U}} = \sum_{\alpha}\phi_{\alpha} dz^{\alpha}\,,\quad\quad\quad \phi_{h}^{*}|_{\mathscr{U}}=\sum_{\alpha}\phi_{\bar{\alpha}}^{*} d{\bar z}^{\alpha}\,,\nonumber
\end{equation}
where 
\begin{equation}
\phi_{\alpha}=\sum_{i,j}\phi_{\alpha j}^{i}e_{i}\otimes e^{j}\,,\quad\quad\quad \phi_{\bar{\alpha}}^{*}=\sum_{i,j}{\phi}_{\bar{\alpha} j}^{*i}e_{i}\otimes e^{j}\,, \nonumber
\end{equation}
the identities $h(\phi_{\bar{\alpha}}^{*}e_{k},e_{l})=h(e_{k},\phi_{\alpha}e_{l})$ become
\begin{equation}
\phi^{*l}_{\bar{\alpha} k} = \sum_{i}h\left(\phi^{*i}_{\bar{\alpha}k}e_{i},e_{l}\right) = \sum_{i}h\left(e_{k},\phi^{i}_{\alpha l}e_{i}\right) = \overline{\phi_{\alpha l}^{k}}\,. \label{adj vs. cc}
\end{equation} 
From (\ref{adj vs. cc}) we get
\begin{eqnarray*}
\overline{[\phi_{\alpha},\phi_{\bar\beta}^{*}]^{i}_{j}} & = & \sum_{k}\overline{\left(\phi_{\alpha k}^{i}\phi^{*k}_{\bar\beta j} - \phi^{*i}_{\bar\beta k}\phi^{k}_{\alpha j}\right)}\\
                                                                          & = & \sum_{k}\left(\phi_{\bar\alpha i}^{*k}\phi_{\beta k}^{j}  -  \phi_{\beta i}^{k}\phi^{*j}_{\bar\alpha k} \right)\\
                                                                          & = & [\phi_{\beta},\phi^{*}_{\bar\alpha}]^{j}_{\,i}\,.
\end{eqnarray*}
Therefore, it follows that
\begin{equation}
\overline{\left({\cal R}^{1,1}_{h}\right)^{i}_{j\alpha\bar\beta}} =  \left({\cal R}^{1,1}_{h}\right)^{j}_{i\beta\bar\alpha}\,, \label{eq:Hermitian-curv-comp.}
\end{equation}
and consequently
\begin{equation}
\overline{\left({\cal K}_{h}\right)^{i}_{j}} =  \left({\cal K}_{h}\right)^{j}_{i}\,.
\end{equation}
In other words ${\cal K}$ is a Hermitian endomorphism of $E$, i.e., $h({\cal K}\cdot,\cdot) = h(\cdot,{\cal K}\cdot)$. This fact will be important through out the article. 
A Hermitian metric $h$ in ${\mathfrak E}$ is said to be \emph{Hermite--Einstein} or \emph{Hermite--Yang--Mills} if ${\cal K}_{h}=cI$, where $I$ denotes here the identity endomorphism of $E$ and $c$ is a constant given by \eqref{val. c}. The celebrated  \emph{Hitchin--Kobayashi correspondence} for Higgs bundles \cite{Hitchin, Simpson} relates the algebraic notion of $\mu$-stability with the differential notion of Hermite--Yang--Mills metric.

\begin{thm}[Hitchin--Kobayashi correspondence \cite{Simpson}]\label{HK-corresp.}
Let ${\mathfrak E}$ be a Higgs bundle over a compact K\"ahler manifold $X$. ${\mathfrak E}$ is $\mu$-polystable if and only if there exists a Hermite-Yang-Mills metric $h$ on ${\mathfrak E}$, i.e., if and only if there exists a Hermitian metric $h$ on ${\mathfrak E}$ satisfying ${\cal K}_{h}=cI$. 
\end{thm}   

The combination of the Killing form and the Cartan involution induced by $h$ on the fibers of $\mathrm{End}\, E$ defines a Hermitian metric on $\mathrm{End}\, E$, and every smooth endomorphism  $M: E\longrightarrow E$ acquires a pointwise norm by $| M |^{2} = \mathrm{tr}\left(M\circ M^{*}_{h}\right)$ where $M^{*}_{h}$ is the Hermitian conjugation of $M$, i.e. $M^{*}_{h}={\overline M}^{\,\rm t}$. In particular, the pointwise norm of ${\cal K}_{h}$ takes the form $\lvert{\cal K}_{h}\lvert^{2} = {\rm tr}({\cal K}_{h}\circ{\cal K}_{h})$. A Higgs bundle is said to have an \emph{approximate Hermite--Yang--Mills metric} if for any $\epsilon>0$, there exists a Hermitian metric $h_{\epsilon}$ such that ${\rm Max}\lvert {\cal K}_{h_{\epsilon}}-cI\lvert < \epsilon$. Theorem \ref{HK-corresp.} has been already extended some years ago to a result relating the $\mu$-semistability and the existence of approximate Hermite--Yang--Mills metrics. Such a result can be stated as follows.

\begin{thm}[\cite{Bruzzo-Granha, Cardona 1},\cite{Li}]\label{HK-corresp. ap.}
Let ${\mathfrak E}$ and $X$ be as in Theorem \ref{HK-corresp.}. ${\mathfrak E}$ is $\mu$-semistable if and only if there exists an approximate Hermite--Yang--Mills metric on it. 
\end{thm}       

More generally, the Hermitian metrics on $X$ and $\mathrm{End}\, E$ induce a Hodge operator 
\begin{equation}
*: A^{p,q}(\textrm{End}\, E) \rightarrow A^{n - q,n - p}(\textrm{End}\, E)\, \nonumber
\end{equation} and a pointwise Hermitian inner product $(\cdot,\cdot)$ on the spaces $A^{p,q}(\mathrm{End}\, E)$, $0 \leq p,q \leq n$ (see \cite{Kobayashi} for details). Such operations define an $L^{2}$-Hermitian inner product on the spaces $A^{p,q}(\mathrm{End}\, E)$ in the two equivalent forms
\begin{equation}
\langle\psi,\eta\rangle = \int\limits_{X}\mathrm{tr}\left(\psi\wedge \overline{*}_{h}\,\eta\right) = \int\limits_{X} (\psi,\eta)\,\frac{\omega^{n}}{n!} \label{eq:L2}
\end{equation}
where the operator 
\begin{equation}
\overline{*}_{h}: A^{p,q}(\textrm{End}\, E) \rightarrow A^{n - p,n - q}(\textrm{End}\, E) \nonumber
\end{equation} 
is the composition of Hermitian conjugation and the Hodge operator, i.e., $\overline{*}_{h}\,\psi = *\left(\psi^{*}_{h}\right)$ for all $\psi$ in $A^{p,q}(\textrm{End}\, E)$. The pointwise norm $|\psi|$ and the induced $L^{2}$-norm $\|\psi\|$ are respectively defined by the standard formulas: 
\begin{equation*}
|\psi|^{2}=(\psi,\psi)\,,\quad\quad\quad \|\psi\|^{2}={\langle\psi,\psi\rangle}
\end{equation*}

Let ${\rm Herm}^{+}({\mathfrak E})$ and ${\rm Herm}({\mathfrak E})$ denote, respectively, the set of Hermitian metrics and Hermitian forms on ${\mathfrak E}$.  ${\rm Herm}^{+}({\mathfrak E})$ is an infinite dimensional manifold whose tangent space at any $h$ can be identified with ${\rm Herm}({\mathfrak E})$ (see \cite{Kobayashi} for details). If $h\in{\rm Herm}^{+}({\mathfrak E})$ is fixed, then every $k\in{\rm Herm}({\mathfrak E})$ defines an endomorphism of $E$, usually denoted by $h^{-1}k$, by imposing the condition 
\begin{equation*}
k(s,s')=h(s,h^{-1}ks')
\end{equation*} 
for all sections $s,s'$ of $E$.\\

%%%%%%%%%%%%%%%%%%%%%%%%%%%%%%%%%%%%%%%%%%%%%%%%%%%%%%%%%%%%%%%%%%%%%%%%%%%%%%%
%%%%%%%%%%%%%%%%%%%%%%%%%%%%%%%%%%%%%%%%%%%%%%%%%%%%%%%%%%%%%%%%%%%%%%%%%%%%%%%

\section{The functional ${\cal J}(h)$ for Higgs bundles}

Let ${\mathfrak E}$ be a Higgs bundle over a compact K\"ahler manifold $X$, and let $h$ be an arbitrary Hermitian metric on ${\mathfrak E}$, considered as a variable. Following the ideas of Kobayashi \cite{Kobayashi}, we introduce a functional given as the $L^{2}$-norm of the Hitchin--Simpson mean curvature 
\begin{equation}
{\cal J}(h) = \frac{1}{2}\int\limits_{X}\lvert{\cal K}_{h}\lvert^{2}\omega^{n} = \frac{n!}{2}\|\mathcal{K}_{h}\|^{2}\,. \label{Def. J}
\end{equation}
We call the functional (\ref{Def. J}) the {\it Kobayashi functional} of the Higgs bundle ${\mathfrak E}$. Notice that by definition it is always non-negative in ${\rm Herm}^{+}(\mathfrak E)$ and coincides with Kobayashi's functional $J(h)$ in \cite{Kobayashi} if $\phi\equiv 0$. Now, let us consider the constant $c$ determined by 
\begin{equation}
rc\int\limits_{X}\omega^{n} = \int\limits_{X}\sigma_{h}\,\omega^{n}\,. \label{def. c}
\end{equation}
By taking the trace of (\ref{Formula R and K}) and using the definitions (\ref{def. deg}) and (\ref{def. K,sigma}) we get 
\begin{eqnarray*}
\int\limits_{X}\sigma_{h}\,\omega^{n} &=& \int\limits_{X}{\rm tr}\,{\cal K}_{h}\,\omega^{n} = \sqrt{-1}n\int\limits_{X}{\rm tr}\,{\cal R}_{h}\wedge\omega^{n-1}\\
                                    &=& 2\pi n\int\limits_{X}c_{1}(E)\wedge\omega^{n-1} = 2\pi n\,{\rm deg}\,{\mathfrak E}
\end{eqnarray*}
where we have used the formulas 
\begin{equation}
{\rm tr}\,{\cal R}_{h}\wedge\omega^{n-1} = {\rm tr}\,R_{h}\wedge\omega^{n-1}\,,\quad\quad\quad  c_{1}(E)=\frac{\sqrt{-1}}{2\pi}{\rm tr}\,R_{h}\,. \nonumber
\end{equation}
Hence from (\ref{def. c}) we get that
\begin{equation}
c = \frac{2\pi n\,{\rm deg}\,{\mathfrak E}}{r\,n!{\rm Vol}\,X} = \frac{2\pi \mu({\mathfrak E})}{(n-1)!{\rm Vol}\,X} \label{val. c}
\end{equation}
which shows that $c$ is independent of $h$ and, in fact, coincides with the constant defined in the Donaldson functional \cite{Kobayashi, Simpson}. Using the  pointwise norm in the space of Hermitian endomorphisms of $E$, we obtain
\begin{equation}
0 \le \lvert {\cal K}_{h} - cI\lvert^{2} = \lvert {\cal K}_{h}\lvert^{2} + rc^{2} - 2c\,\sigma_{h}\, \label{Key ineq.}
\end{equation}
where $I$ denotes the identity endomorphism in $E$. The inequality (\ref{Key ineq.}) generalizes the inequality in \cite{Kobayashi} involving the Chern mean curvature $K_{h}$
to the Hitchin--Simpson mean curvature ${\cal K}_{h}$. As in the holomorphic vector bundle case, it can be used to find a lower bound condition for ${\cal J}(h)$. In fact, integrating (\ref{Key ineq.}) and using (\ref{def. c}) we get     
\begin{equation}
\int\limits_{X}\lvert {\cal K}_{h}\lvert^{2}\omega^{n} \ge 2c\int\limits_{X}\sigma_{h}\,\omega^{n} - rc^{2}\int\limits_{X}\omega^{n} = rc^{2}\int\limits_{X}\omega^{n}\,.  \label{Key ineq. 2}
\end{equation}
From the inequality (\ref{Key ineq. 2}) and using the definition (\ref{val. c}) we conclude that 
\begin{equation}
{\cal J}(h) \ge \frac{rc^{2}}{2}\int\limits_{X}\omega^{n} = \frac{2n(\pi\,{\rm deg\,}{\mathfrak E})^{2}}{r(n-1)!\,{\rm Vol\,}X}={\cal C}\,. \label{Bounded cond.}
\end{equation}

Notice in particular that the nonnegative lower bound ${\cal C}$ above is different from the constant $c$ previously defined. We conclude from (\ref{Key ineq.}) that the equality in (\ref{Bounded cond.}) follows  if and only if we have that ${\cal K}_{h}=cI$. Therefore, in strict analogy with the classical case of holomorphic vector bundles we conclude the following result.

\begin{thm}\label{Bounded Thm}
Let ${\mathfrak E}$ be a Higgs bundle over a compact K\"ahler manifold $X$. The functional ${\cal J}(h)$ is bounded below as in (\ref{Bounded cond.}) by a constant ${\cal C}$ which depends on invariants of the bundle and the cohomology class of the K\"ahler form in $X$. Moreover, ${\cal J}(h)$ attains this lower bound at $h=h_{0}$ if and only if $h_{0}$ is a  Hermite--Yang--Mills metric of 
${\mathfrak E}$.
\end{thm}

Notice that even though ${\cal J}(h)$ and $J(h)$ are different functionals, the lower bound ${\cal C}$ is the same as the lower bound associated to $J(h)$. Now, from Theorem \ref{Bounded Thm} and inequality (\ref{Key ineq.}) it is clear that if we have an approximate Hermite--Yang--Mills metric on ${\mathfrak E}$, the constant ${\cal C}$ becomes an infimum for the functional ${\cal J}(h)$. Putting together Theorem \ref{Bounded Thm} and Theorems \ref{HK-corresp.} and \ref{HK-corresp. ap.} we conclude that the existence of a minimum or infimum for the Kobayashi functional depends essentially on algebraic conditions of ${\mathfrak E}$. More precisely, we have the following results.

\begin{cor}
Let ${\mathfrak E}$ and $X$ be as in Theorem \ref{Bounded Thm}. Then, ${\cal J}(h)$ attains the lower bound ${\cal C}$ given by (\ref{Bounded cond.}) if and only if ${\mathfrak E}$ is $\mu$-polystable.  
\end{cor}
\begin{cor}
Let ${\mathfrak E}$ and $X$ be as in Theorem \ref{Bounded Thm}. Then, the constant ${\cal C}$ given by (\ref{Bounded cond.}) is an infimum of ${\cal J}(h)$ if and only if ${\mathfrak E}$ is $\mu$-semistable.  
\end{cor}

The above corollaries immediately link boundedness properties of ${\cal J}(h)$ with stability properties of the Higgs bundle, e.g., since $\mu$-stability is an special case of $\mu$-polystability or any tensor product of $\mu$-polystable Higgs bundles is $\mu$-polystable, then it is clear that the functional ${\cal J}(h)$ associated to any $\mu$-stable Higgs bundle or any tensor product of $\mu$-polystable Higgs bundles necessarily attains its lower bound. Also, if ${\mathfrak E}$ is $\mu$-semistable but not $\mu$-polystable, then ${\cal J}(h)$ can be arbitrarily close to ${\cal C}$ (but it never reaches such a value). We won't address such algebraic aspects in this article, and instead refer the reader to \cite{Biswas-Schumacher, Bruzzo-Granha, Cardona 1, Cardona 2} or \cite{Simpson}.\\

On a holomorphic vector bundle $E$, Kobayashi \cite{Kobayashi} defines the functional $I(h)$ as the $L^{2}$-norm of the curvature of the induced Chern connection of $h$, and shows that the difference $I(h) - J(h)$ is a constant term involving the K\"ahler form of $X$ and the first and second Chern classes of $E$. The present situation is slightly more complicated, given that the curvature form of the Hitchin--Simpson connection possesses extra terms of type (2,0) and (0,2) which should be considered when defining its $L^{2}$-norm. However, its  $(1,1)$-part  also specializes to $R_{h}$ when $\phi\equiv 0$, and moreover, it follows from \eqref{eq:Hermitian-curv-comp.} that ${\cal R}^{1,1}_{h}$ still takes values in $(\mathrm{End}\,E)_{+}$. Hence, the pointwise inner product for ${\cal R}^{1,1}_{h}$ is explicitly defined with respect to any unitary frame in $E$ and local holomorphic coordinates on a neighborhood 
$\mathscr{U}\subset X$ by the formula
\begin{equation}
\left.\left|{\cal R}^{1,1}_{h}\right|^{2}\right|_{\mathscr{U}} = \sum_{\alpha,\beta,\gamma,\delta}\sum_{i,j} g^{\alpha\bar\beta}g^{\gamma\bar\delta}({\cal R}_{h})^{i}_{\,j\alpha\bar\delta}({\cal R}_{h})^{j}_{\,i\gamma\bar\beta}\,, \label{R11 components}
\end{equation}
where $({\cal R}_{h})^{i}_{\,j\alpha\bar\delta}$ in an abbreviation for the components of ${\cal R}_{h}^{1,1}$ defined in (\ref{Comp. cal R11}). In terms of the $L^{2}$-inner product \eqref{eq:L2}, we construct the functional ${\cal I}(h)$ in ${\rm Herm}^{+}(\mathfrak E)$ as the $L^{2}$-energy
\begin{equation}
{\cal I}(h) = \frac{1}{2}\int\limits_{X}\left|{\cal R}^{1,1}_{h}\right|^{2}\omega^{n}\,. \label{Def. I}
\end{equation} 
\begin{remark}
In analogy to the holomorphic vector bundle case, the functional ${\cal I}(h)$ coincides with the Yang-Mills-Higgs functional with a choice of holomorphic structure and holomorphic Higgs field\footnote{Over a $C^{\infty}$ Hermitian vector bundle there is a one-to-one correspondence between unitary connections on $E$  with curvature of type $(1,1)$ and $d''_{E}$-operators (see \cite{Lubke} for details). In the Higgs bundle setting, by fixing  $d''_{E}$ and a holomorphic Higgs field $\phi$, the remaining degrees of freedom of the Yang--Mills-Higgs functional are parametrized by Hermitian metrics $h$. The resulting functional is what we consider here.} 
\cite{Bradlow-Wilkin, Wentworth} only in the special (but mostly relevant) case when $n = 1$, since the latter is instead defined as the $L^{2}$-energy of the full Hitchin--Simpson curvature form. Moreover, in the case when $n = 1$ it also follows from \eqref{Formula R and K} that 
\begin{equation}
\sqrt{-1}\mathcal{R}_{h} = \sqrt{-1}\mathcal{R}^{1,1}_{h} = \mathcal{K}_{h}\omega\,, \nonumber
\end{equation}
hence $|\mathcal{R}^{1,1}_{h}|^{2} = |\mathcal{K}_{h} |^{2}$ and consequently we also conclude that ${\cal I}(h) = {\cal J}(h)$. Therefore, from now on we will assume that $n \geq 2$.
\end{remark}

In order to estimate the difference ${\cal I}(h) - {\cal J}(h)$ we will require the following auxiliary result.

\begin{lem}\label{lemma:tr-curv-sqr} Let $(\mathfrak{E},h)$ be a Hermitian Higgs bundle over a K\"ahler manifold $(X,\omega)$ of dimension $n \geq 2$. Then, the $(1,1)$-part of the curvature form of its  Hitchin--Simpson connection satisfies
\begin{equation}
{\rm tr}\left({\cal R}^{1,1}_{h}\wedge{\cal R}^{1,1}_{h}\right)\wedge\omega^{n-2} = \frac{1}{n(n-1)}\left(\lvert{\cal R}^{1,1}_{h}\lvert^{2} -  \lvert {\cal K}_{h}\lvert^{2}\right)\omega^{n}\,.
\end{equation}
\end{lem}

\begin{proof}
As in \cite{Kobayashi, Siu}, let $\mathscr{U}\subset X$ be an local neighborhood with a local unitary coframe $\{\theta^{\alpha}\}_{\alpha=1}^{n}$, and $\{e_{i}\}_{i=1}^{r}$ be a local unitary frame for $E|_{\mathscr{U}}$ with dual frame $\{e^{i}\}_{i=1}^{r}$, so that  
\begin{eqnarray*}
g|_{\mathscr{U}} = \sum_{\alpha}\theta^{\alpha}\otimes\bar{\theta}^{\alpha}\,,\quad\quad \omega|_{\mathscr{U}} = \sqrt{-1}\sum_{\alpha}\theta^{\alpha}\wedge\theta^{\bar\alpha}\,,\quad\quad h |_{\mathscr{U}} = \sum_{i}e^{i}\otimes\bar{e}^{i}\,.
\end{eqnarray*}
The frame $\{\theta^{\alpha}\}_{\alpha=1}^{n}$ satisfies the condition: 
\begin{equation}\label{Prop. u. frames}
n(n-1)\,\theta^{\alpha}\wedge\bar\theta^{\beta}\wedge\theta^{\gamma}\wedge\bar\theta^{\delta}\wedge\omega^{n-2} =
\left\{
\begin{array}{rl}
 -\omega^{n} & \quad\text{if} \quad\alpha = \beta \neq \gamma = \delta\,,\\\\
 \omega^{n} & \quad\text{if} \quad\alpha = \delta \neq \beta = \gamma\,,\\\\
0\;\; & \quad\text{otherwise.}
\end{array}\right.
\end{equation}
The curvature component ${\cal R}^{1,1}_{h}$ can be written in terms of these frames as
\begin{equation}
{\cal R}^{1,1}_{h} = \sum_{\alpha,\beta}({\cal R}_{h})_{\alpha\bar\beta}\,\theta^{\alpha}\wedge\bar\theta^{\beta}\,, 
\qquad ({\cal R}_{h})_{\alpha\bar\beta} = \sum_{i,j}({\cal R}_{h})^{i}_{j\alpha\bar\beta}\,e_{i}\otimes e^{j}\,,
\end{equation}
and using (\ref{Comp. cal K}) we obtain 
\begin{equation}
({\cal K}_{h})^{i}_{j} = \sum_{\alpha,\beta}g^{\alpha\bar\beta}({\cal R}_{h})^{i}_{j\alpha\bar\beta} = \sum_{\alpha}({\cal R}_{h})^{i}_{j\alpha\bar\alpha}  \nonumber
\end{equation}
where the second equality in the above expression follows from the condition $g^{\alpha\bar\beta}=\delta^{\alpha\bar\beta}$. Hence, we have that
\begin{equation}
\lvert {\cal K}_{h}\lvert^{2} = {\rm tr}({\cal K}_{h}\circ{\cal K}_{h}) = \sum_{i,j}({\cal K}_{h})^{i}_{j}({\cal K}_{h})^{j}_{i} =  \sum_{i,j,\alpha,\beta}({\cal R}_{h})^{i}_{j\alpha\bar\alpha}({\cal R}_{h})^{j}_{i\beta\bar\beta}\,. \label{Norm K, comp.}
\end{equation}
Using the property (\ref{Prop. u. frames}) and definitions (\ref{Norm K, comp.}) and (\ref{R11 components}) we get the identity
\begin{eqnarray*}
{\rm tr}\left({\cal R}^{1,1}_{h}\wedge{\cal R}^{1,1}_{h}\right)\wedge\omega^{n-2} &=& \sum_{\alpha,\beta,\gamma,\delta}\sum_{i,j} ({\cal R}_{h})^{i}_{\,j\alpha\bar\beta}({\cal R}_{h})^{j}_{\,i\gamma\bar\delta}\,\theta^{\alpha}\wedge\bar\theta^{\beta}\wedge\theta^{\gamma}\wedge\bar\theta^{\delta}\wedge\omega^{n-2}\\
                                                       &=& \frac{1}{n(n-1)}\sum_{i,j,\alpha,\beta} ({\cal R}_{h})^{i}_{\,j\alpha\bar\beta}({\cal R}_{h})^{j}_{\,i\beta\bar\alpha}\,\omega^{n}\\
                                                       && -\, \frac{1}{n(n-1)}\sum_{i,j,\alpha,\beta} ({\cal R}_{h})^{i}_{\,j\alpha\bar\alpha}({\cal R}_{h})^{j}_{\,i\beta\bar\beta}\,\omega^{n}\\
                                                       &=& \frac{1}{n(n-1)}\left(\lvert{\cal R}^{1,1}_{h}\lvert^{2} -  \lvert {\cal K}_{h}\lvert^{2}\right)\omega^{n}\,.
\end{eqnarray*}
\end{proof}

\begin{pro}\label{Formula I vs. J}
Let ${\mathfrak E}$ be a Higgs bundle over a compact K\"ahler manifold $X$ of dimension $n \geq 2$, and let ${\cal J}(h)$ and ${\cal I}(h)$ be the functionals on 
${\mathfrak E}$ defined by (\ref{Def. J}) and (\ref{Def. I}), resp. Then
\begin{eqnarray*}
{\cal I}(h) - {\cal J}(h) &=& 2\pi^{2}n(n-1)\int\limits_{X}\left(2\,c_{2}(E,h)-c_{1}(E,h)^{2}\right)\wedge\omega^{n-2}\\
 && +\, n(n-1)\int\limits_{X}{\rm tr}\left({\cal R}^{1,1}_{h}\wedge [\phi,\phi^{*}_{h}]\right)\wedge\omega^{n-2}\,.
\end{eqnarray*}
\end{pro}

\begin{proof}
We note first that by using the cyclic property of the trace and the conditions 
$\phi\wedge\phi=0$ and $\phi^{*}_{h}\wedge\phi^{*}_{h} = 0$, it follows that 
\begin{eqnarray*}
{\rm tr}\left([\phi,\phi^{*}_{h}]^{2}\right) &=& {\rm tr}\left(\phi\wedge\phi^{*}_{h}\wedge \phi\wedge\phi^{*}_{h} + \phi^{*}_{h}\wedge\phi\wedge\phi^{*}_{h}\wedge\phi\right)\\
                                                             &=&  {\rm tr}\left(\phi\wedge\phi^{*}_{h}\wedge \phi\wedge\phi^{*}_{h}\right) - {\rm tr}\left(\phi\wedge\phi^{*}_{h}\wedge\phi\wedge\phi^{*}_{h}\right) =0
\end{eqnarray*}
hence
\begin{eqnarray*}
{\rm tr}\left({\cal R}^{1,1}_{h}\wedge{\cal R}^{1,1}_{h}\right) &=& {\rm tr}\left(R_{h}^{2}+ [\phi,\phi^{*}_{h}]\wedge R_{h} + R_{h}\wedge[\phi,\phi^{*}_{h}] + [\phi,\phi^{*}_{h}]^{2}\right)\\
                                                                                &=& {\rm tr}\left(R_{h}^{2}\right) + 2\,{\rm tr}\left(R_{h}\wedge [\phi,\phi^{*}_{h}]\right) \,. 
\end{eqnarray*}
The first term in the right-hand side is equal to 
\begin{equation}
{\rm tr}(R_{h}^{2}) = 4\pi^{2}(2\,c_{2}(E,h)-c_{1}(E,h)^{2})\,. \label{Top. Invariant}
\end{equation}
Now, in the second term we can replace again $R_{h}$ by ${\cal R}^{1,1}_{h}$ and therefore, using (\ref{Top. Invariant}), we get 
\begin{equation}
{\rm tr}\left({\cal R}^{1,1}_{h}\wedge{\cal R}^{1,1}_{h}\right) = 4\pi^{2}\left(2\,c_{2}(E,h)-c_{1}(E,h)^{2}\right) + 2\,{\rm tr}\left({\cal R}^{1,1}_{h}\wedge [\phi,\phi^{*}_{h}]\right). \label{Key formula 1}
\end{equation}
The result then follows from Lemma \ref{lemma:tr-curv-sqr} after wedging the identity (\ref{Key formula 1}) with $\omega^{n-2}$ and integrating over $X$.
\end{proof}

\begin{remark}
In the case when $\phi\equiv0$, the identity in Proposition \ref{Formula I vs. J} reduces to a formula of Kobayashi \cite{Kobayashi}, namely
\begin{equation}
I(h) - J(h) = 2\pi^{2}n(n-1)\int\limits_{X}(2\,c_{2}(E,h)-c_{1}(E,h)^{2})\wedge\omega^{n-2}\,. \label{Classical formula}
\end{equation}
On the other hand, in the particular case $n={\dim\,}X=2$, the identity in Proposition \ref{Formula I vs. J} does not depend on $\omega$, although in general, it now depends explicitly on $h$ through the additional term
\[
n(n-1)\int\limits_{X}{\rm tr}\left({\cal R}^{1,1}_{h}\wedge [\phi,\phi^{*}_{h}]\right)\wedge\omega^{n-2}.
\]
This fact marks a fundamental difference between the functionals ${\cal J}(h)$ and ${\cal I}(h)$, as opposed to the holomorphic vector bundle case. Since the right-hand side of \eqref{Classical formula} is a topological invariant, $J(h)$ is essentially the same as $I(h)$, which is the Yang--Mills functional on a holomorphic gauge. Moreover, since ${\cal I}(h)$ is just a component of the Yang-Mills-Higgs functional, in the Higgs bundle case we also conclude that ${\cal J}(h)$ is fundamentally different from the Yang--Mills--Higgs functional. 
\end{remark}

%%%%%%%%%%%%%%%%%%%%%%%%%%%%%%%%%%%%%%%%%%%%%%%%%%%%%%%%%%%%%%%%%%%%%%%%%%%%%
%%%%%%%%%%%%%%%%%%%%%%%%%%%%%%%%%%%%%%%%%%%%%%%%%%%%%%%%%%%%%%%%%%%%%%%%%%%%%

\section{Variational properties of ${\cal J}(h)$}

We will now consider arbitrary 1-parameter families in the space ${\rm Herm}^{+}({\mathfrak E})$ of Hermitian metrics of ${\mathfrak E}$ in order to study some evolution properties of ${\cal J}(h)$. For $\epsilon >0$, $-\epsilon\le t\le \epsilon$, let $h_{t}$ be a 1-parameter family in ${\rm Herm}^{+}({\mathfrak E})$ with $h_{0} = h$. The tangent vectors $\partial_{t}h_{t}$ then define a 1-parameter family in ${\rm Herm}({\mathfrak E})$. Let us denote $\phi^{*}_{t} = \phi^{*}_{h_{t}}$, $\mathcal{R}^{1,1}_{h_{t}} = \mathcal{R}^{1,1}_{t}$ and let $v_{t} = h_{t}^{-1}\partial_{t}h_{t}$ be the associated 1-parameter family of Hermitian endomorphisms of $E$. Then we obtain that $\partial_{t}\phi^{*}_{t}=[\phi^{*}_{t},v_{t}]$ (see \cite{Cardona 1} for a proof of this). From (\ref{Formula eff. HS curv.}) we obtain
\begin{equation}
\partial_{t}{\cal R}^{1,1}_{t} = \partial_{t}R_{t} + [\phi,[\phi^{*}_{t},v_{t}]]  \nonumber
\end{equation}
and by applying the operator $\sqrt{-1}\Lambda$ on this expression we get a formula for the variation of the mean curvature ${\cal K}_{t}$, namely,
\begin{equation}
\partial_{t}{\cal K}_{t} = \partial_{t}K_{t} + \sqrt{-1}\Lambda[\phi,[\phi^{*}_{t},v_{t}]]\,. \label{Var. Mean Curv. R}
\end{equation}
Let us consider 
\begin{equation}
v = v_{t}|_{t = 0} = h_{t}^{-1}\partial_{t}h_{t}|_{t=0}\,. \label{Def. v}
\end{equation} 
The standard formula $\partial_{t}K_{t}|_{t=0} = \sqrt{-1}\Lambda D''D'_{h}v$ can be derived using coordinates\footnote{In fact, in \cite{Kobayashi}, Ch. IV, it is proved that $\partial_{t}(K_{t})^{i}_{j}\big\lvert_{t=0}=-\sum g^{\alpha\bar\beta}v^{i}_{\,j\alpha\bar\beta}$ where $v^{i}_{\,j\alpha\bar\beta}$ are interpreted as minus the coefficients of $D''D'_{h}v$, consequently $-\sum g^{\alpha\bar\beta}v^{i}_{\,j\alpha\bar\beta}$ are exactly the components of the endomorphism 
$\sqrt{-1}\Lambda D''D'_{h}v$.} and hence from (\ref{Var. Mean Curv. R}) we obtain     
\begin{equation}
\left.\partial_{t}{\cal K}_{t}\right|_{t=0} =  \sqrt{-1}\Lambda D''D'_{h}v+ \sqrt{-1}\Lambda[\phi,[\phi^{*}_{h},v]]\,. \label{Var. HS mean curv.}
\end{equation}
On the other hand, it follows from the definitions of ${\cal D}'_{h}$ and ${\cal D}''$ that
\begin{eqnarray*}
{\cal D}''{\cal D}'_{h}v &=& {\cal D}''(D'_{h}v + [\phi^{*}_{h},v])\\ 
                            &=& D''D'_{h}v + [\phi,D'_{h}v] + D'' [\phi^{*}_{h},v] + [\phi,[\phi^{*}_{h},v]]
\end{eqnarray*}
and by applying the $\sqrt{-1}\Lambda$ operator on the above expression, we obtain exactly the right-hand side of (\ref{Var. HS mean curv.}), that is
\begin{equation}
\left.\partial_{t}{\cal K}_{t}\right|_{t=0} = \sqrt{-1}\Lambda {\cal D}''{\cal D}'_{h}v\,, \label{Key formula 11}
\end{equation}
which can be seen as a natural extension to Hermitian Higgs bundles of the formula of Kobayashi for the first variation of the Chern mean curvature of a Hermitian vector bundle. 

\begin{lem}\label{lemma:formula D'D''}
Let $(\mathfrak{E},h)$ be a Hermitian Higgs bundle over a compact K\"ahler manifold $X$ and $v$ the endomorphism defined in (\ref{Def. v}), then we have 
\begin{equation}
\langle\sqrt{-1}\Lambda\mathcal{D}''\mathcal{D}'_{h}v,\mathcal{K}_{h}\rangle = \langle\mathcal{D}'_{h}v,\mathcal{D}'_{h}\mathcal{K}_{h}\rangle\,. \label{eq:identity-inner-prod}
\end{equation}
\end{lem}
\begin{proof} 
By definition of the $L^{2}$-inner product, the right-hand side of \eqref{eq:identity-inner-prod} is equal to
\[
\langle D'_{h}v,D'_{h}\mathcal{K}_{h}\rangle + \langle [\phi^{*}_{h},v],[\phi^{*}_{h},\mathcal{K}_{h}]\rangle\,.
\]
On the other hand, we have seen that 
\begin{equation}
\sqrt{-1}\Lambda\mathcal{D}''\mathcal{D}'_{h}v = \sqrt{-1}\Lambda D''D'_{h}v +  \sqrt{-1}\Lambda[\phi,[\phi^{*}_{h},v]]\,. \nonumber
\end{equation} 
Hence, the left-hand side of \eqref{eq:identity-inner-prod} is equal to
\[
\langle \sqrt{-1}\Lambda D''D'_{h}v, \mathcal{K}_{h}\rangle + \langle \sqrt{-1}\Lambda[\phi,[\phi^{*}_{h},v]],\mathcal{K}_{h}\rangle\,.
\]
Now, the identity $\langle\sqrt{-1}\Lambda D''D'_{h}v,K\rangle = \langle D'_{h}v,D'_{h} K\rangle$ is proved in \cite{Kobayashi}. Hence, to prove the lemma, it is sufficient to prove that the following Hermitian inner products coincide
\begin{equation}
\langle \sqrt{-1}\Lambda[\phi,[\phi^{*}_{h},v]],\mathcal{K}_{h}\rangle =  \langle [\phi^{*}_{h},v],[\phi^{*}_{h},\mathcal{K}_{h}]\rangle\,.  \label{eq:identity K}
\end{equation}
The cyclic property of the trace and the Hermitian nature of $\mathcal{K}_{h}$ imply that the local expressions $\phi |_{\mathcal{U}} = \sum \phi_{\alpha} d z^{\alpha}$ and $\phi^{*}_{h} |_{\mathcal{U}} = \sum \phi^{*}_{\bar{\alpha}} d \bar{z}^{\alpha}$ satisfy the following identity
\[
\sum_{\alpha,\beta}\mathrm{tr}\left(g^{\alpha\bar{\beta}}[\phi_{\alpha},[\phi^{*}_{\bar{\beta}},v]]\mathcal{K}_{h}\right) = \sum_{\alpha,\beta}\mathrm{tr}\left( g^{\alpha\bar{\beta}}[\phi^{*}_{\bar{\beta}},v] \left[\phi^{*}_{\bar{\alpha}},\mathcal{K}_{h}\right]^{*}_{h}\right).
\]
which are precisely the density terms in the Hermitian inner products \eqref{eq:L2} for each of the terms in \eqref{eq:identity K}. Then the identity \eqref{eq:identity K} follows.
\end{proof}

\begin{pro}\label{Ev. Pro.}
Let ${\mathfrak E}$ be a Higgs bundle over a compact K\"ahler manifold $X$ and let $h_{t}$, with $-\epsilon \le t\le \epsilon$, be a 1-parameter family in ${\rm Herm}^{+}({\mathfrak E})$ such that $h_{0} = h$. Then 
\begin{equation}
\partial_{t}{\cal J}(h_{t})\lvert_{t=0} = \langle\mathcal{D}'_{h}v,\mathcal{D}'_{h}{\cal K}_{h}\rangle \,. \label{eq:critical-points}
\end{equation} 
\end{pro}

\begin{proof}
It follows from (\ref{Key formula 11}) that
\[
\partial_{t}\lvert{\cal K}_{t}\lvert^{2}\big\lvert_{t=0} = 2\,{\rm tr}(\partial_{t}{\cal K}_{t}\circ{\cal K}_{t})\big\lvert_{t=0} = 2\,{\rm tr}(\sqrt{-1}\Lambda {\cal D}''{\cal D}'_{h}v\circ{\cal K}_{h})\,,
\]
from which it readily follows that
\[
\partial_{t}{\cal J}(h_{t})\lvert_{t=0} = \langle\sqrt{-1}\Lambda\mathcal{D}''\mathcal{D}'_{h}v,\mathcal{K}_{h}\rangle\,.
\]
Proposition \ref{Ev. Pro.} is then a consequence of Lemma \ref{lemma:formula D'D''}.
\end{proof}

\begin{remark}
Equation \eqref{eq:critical-points}  in Proposition \ref{Ev. Pro.} is a generalization of the equation for the critical points of the functional $J(h)$. In fact, if $\phi\equiv 0$ the second term vanishes and we have 
\begin{equation}
\partial_{t}J(h_{t})\lvert_{t=0} = \langle D'_{h}v,D'_{h}{\cal K}_{h}\rangle \nonumber
\end{equation}
which is a key step in \cite{Kobayashi} to prove that the critical points of $J(h)$ are the Hermitian metrics whose Chern mean curvature is parallel with respect to the induced Chern connection. In analogy to \cite{Kobayashi}, we conclude the following result.
\end{remark}

\begin{thm}
Let ${\mathfrak E}$ be a Higgs bundle over a compact K\"ahler manifold $X$ and let ${\cal J}(h)$ be the functional on ${\mathfrak E}$ defined by (\ref{Def. J}). Then, the Euler--Lagrange equations of $\mathcal{J}(h)$ are 
\begin{equation}
\mathcal{D}_{h}\mathcal{K}_{h} =0\,.
\end{equation}
Consequently, a Hermitian metric on $\mathfrak{E}$ is a critical point of $\mathcal{J}(h)$ if and only if the corresponding Hitchin--Simpson mean curvature is parallel with respect to the Hitchin--Simpson connection. 
\end{thm}
\begin{proof} 
Since $\mathcal{K}_{h}$ takes values in $\mathrm{End}\, E_{-}$, it follows that $\left(\mathcal{D}'_{h}\mathcal{K}_{h}\right)^{*} = \mathcal{D}''\mathcal{K}_{h}$. Hence $\mathcal{D}'_{h}\mathcal{K}_{h} = 0$ if and only if $\mathcal{D}_{h}\mathcal{K}_{h} = 0$. Now, for the special linear 1-parameter family of the form $h_{t} = h + t\mathcal{K}_{h}$ with $t$ sufficiently small, it follows that $v = \mathcal{K}_{h}$. From this and (\ref{eq:critical-points}) it follows that $\partial_{t}{\cal J}(h_{t})\lvert_{t=0} = 0$ if and only if $\mathcal{D}'_{h}\mathcal{K}_{h} = 0$, and the claim follows.
\end{proof}

In conclusion, we can see that most of our results (with the exception of Proposition \ref{Formula I vs. J}) are natural extensions of well-known properties of the functional $J(h)$ in the holomorphic vector bundle case. However, it is still an outstanding problem to perform a thorough study of the flow properties of $\mathcal{J}(h)$ near critical points. We plan to return to such a question in a forthcoming article. 

%%%%%%%%%%%%%%%%%%%%%%%%%%%%%%%%%%%%%%%%%%%%%%%%%%%%%%%%%%%%%%%%%%%%%%%%%%
%%%%%%%%%%%%%%%%%%%%%%%%%%%%%%%%%%%%%%%%%%%%%%%%%%%%%%%%%%%%%%%%%%%%%%%%%%

\end{document}